\documentclass[11pt]{amsart}
\usepackage{amsfonts, amstext, amsmath, amsthm, amscd, amssymb, tikz-cd}

\usepackage{graphicx}
\usepackage{listings}

\usepackage{geometry}
\usepackage{graphics, pinlabel, color}
\usepackage{comment}
\usepackage{enumerate}

\numberwithin{equation}{section}

\theoremstyle{plain}
\newtheorem{theorem}[equation]{Theorem}
\newtheorem{lemma}[equation]{Lemma}
\newtheorem{proposition}[equation]{Proposition}

\theoremstyle{definition}

\newtheorem{remark}[equation]{Remark}

\newtheorem{definition}[equation]{Definition}

\newtheorem{question}[equation]{Question}

\begin{document}

\title{Co-rank of weakly parafree $3$-manifold groups}

\author{Shelly Harvey$^{\dag}$}
\address{Rice University, Houston, TX, USA}
\email{shelly@rice.edu}

\author{Eamonn Tweedy}
\address{Widener University, Philadelphia, PA, USA}
\email{etweedy@widener.edu}

\thanks{$^{\dag}$ The author was partially supported by National Science Foundation grants DMS-1309070 and DMS-1613279 and by a grant from the Simons Foundation (\#304538 to Shelly Harvey).}


\begin{abstract}
Recall that a group is called large if it has a finite index subgroup which surjects onto a non-abelian free group.  By work of Agol and Cooper-Long-Reid, most $3$-manifold groups are large; in particular, the fundamental groups of hyperbolic $3$-manifolds are large.  In previous work, the first author gave examples of closed, hyperbolic $3$-manifolds with arbitrarily large first homology rank but whose fundamental groups do not surject onto a non-abelian free group.  We call a group very large if it surjects onto a non-abelian free group.   In this paper, we consider the question of whether the groups of homology handlebodies -  which are very close to being free - are very large.  We show that the fundamental group of W. Thurston's tripus manifold, is not very large; it is known to be weakly parafree by Stallings' Theorem and large by the work of Cooper-Long-Reid since the tripus is a hyperbolic manifold with totally geodesic boundary.   It is still unknown if a $3$-manifold group that is weakly parafree of rank at least $3$ must be very large.  However, we more generally consider the co-rank of the fundamental group, also known as the cut number of the manifold.  For each integer $g \geq 1$ we construct a homology handlebody $Y_g$ of genus $g$ whose group has co-rank equal to $r(g)$, where $r(g)=g/2$ for $g$ even and $r(g)=(g+1)/2$ for $g$ odd.  That is, these groups are weakly parafree of rank $g$ and surject onto a free group of rank roughly half of $g$ but no larger.  
\end{abstract}

\maketitle


\section{Introduction}

Let $M$ be a compact, connected, orientable $3$-dimensional manifold.  If $M$ is irreducible and has non-empty incompressible boundary and $M$ is not covered by a $S^1 \times S^1 \times I$, then it was shown by D. Cooper, D. Long, and A. Reid \cite{CLR} that $\pi_1(M)$ is large.   Recall that a group is said to be large if it has a finite index subgroup that maps onto a non-abelian free group.  In addition, using work of F. Haglund and D. Wise, I. Agol\footnote{Agol proved the case when $M$ is hyperbolic; the other cases were known.  See \cite{AFW} for more details. } \cite{Agol} showed that if $M$ is closed and irreducible and $\pi_1(M)$ is not finite or solvable then $\pi_1(M)$ is large (see also \cite[Diagram 4]{AFW}).  As a result, we see that most $3$-manifold groups are large.  One could then ask what is the minimal size of the index one needs or if even if $\pi_1(M)$ itself maps onto a non-abelian free group. 

\begin{definition} Let $G$ be a group.  We say that $G$ is \textbf{very large} if it admits a surjective homomorphism onto a non-abelian free group. \end{definition}

In 2002, the first author gave examples of closed hyperbolic $3$-manifold groups with arbitrarily large first homology rank but that are not very large. By Agol's theorem, these examples are all large. 

\newtheorem*{theoremHa}{Theorem (\cite{Ha02})}
\begin{theoremHa} For each $n\geq 1$, there is a closed, orientable, hyperbolic $3$-dimensional manifold $M_n$ such $\beta_1(M_n)=n$ and $\pi_1(M_n)$ is not very large. 
\end{theoremHa}

In fact, it was shown that the groups cannot even map onto $F/F_4$ where $F$ is the free group of rank $2$ and $F_4$ is the $4^{th}$ term of the lower central series of $F$ \cite[Proposition 3.3]{Ha02}.  See also work by C. Leinger and A. Reid \cite{LeRe02}, A. Sikora \cite{Si05}), and I. Gelbukh \cite{G15,G17}.

In this article, we are interested in the fundamental groups of $3$-dimensional homology handlebodies.  Their fundamental groups are very close to being free; they are weakly parafree, and hence the question of whether they are  very large becomes much more subtle and interesting question.   Recall that $M$ is a homology handlebody of genus $g$ if it has the same homology groups as a $3$-dimensional handlebody of genus $g$; that is, $\tilde{H}_1(M)\cong \mathbb{Z}^g$ and $\tilde{H}_i(M)=0$ for $i\neq 1$.

\begin{remark} If $M$ is a homology handlebody of genus $g$ and $G=\pi_1(M)$ then by Stallings' Theorem \cite{Stallings}, $G$ is weakly parafree of rank $g$; that is, 
	\[F(g)/F(g)_k \cong G/G_k\]
	for all $k\geq 1$ where $F(g)$ is the free group of rank $g$ and $G_k$ is the $k^{th}$ term of the lower central series of $G$.
\end{remark}

Thus, groups of homology handlebodies of genus $g$ look quite similar to free groups of large rank as $g$ increases.  In particular, one cannot hope to use the lower central series quotients to obstruct being very large.   

It was shown in \cite[p. 44]{Ker01} (see also Proposition 2.2 in \cite{Ha02}) that if the fundamental group of $M$ is very large then there is an infinite cyclic cover whose first homology has positive rank as a left $\mathbb{Z}[t^{\pm 1}]$-module.  The examples in \cite{Ha02} were shown to have torsion (as a $\mathbb{Z}[t^{\pm 1}]$ module) in first homology for every infinite cyclic cover and hence their groups could not be very large.  However, this obstruction also fails for homology handlebodies.  In fact, the rank of the first homology of any poly-torsion-free-abelian covering space of $M$ is maximal.  Recall that a group $\Gamma$ is poly-torsion-free-abelian (PTFA) if it admits
a normal series $\left\{  1\right\}  =G_{0}\vartriangleleft
G_{1}\vartriangleleft \cdots\vartriangleleft G_{n}=\Gamma$ such
that each of the factors $G_{i+1}/G_{i}$ is torsion-free abelian.  If $A$ is a finitely generated left module over $\mathbb{Z}\Gamma$ with $\Gamma$ PTFA then it has a well defined rank as a $\mathbb{Z}\Gamma$-module.  Note that any free abelian group is PTFA.  

\begin{remark}Let $\Gamma$ be a PTFA group. If $M$ is a $3$-dimensional homology handlebody of genus $g$, then by \cite[Lemma 2.12]{COT} for any non-trivial homomorphism $\phi:\pi_1(M) \rightarrow \Gamma$, the covering space $M_\phi$ associated to $\phi$ satisfies
	\[rank_{\mathbb{Z}\Gamma} H_1(M_\phi)=g-1.\]
\end{remark}

We show in Lemma~\ref{lem:summand}, that if $\pi_1(M)$ is very large then there is a $\mathbb{Z}^2$ covering space of $M$, $M_\phi$, such that  $H_1(M_\phi)$ has a $\mathbb{Z}[\mathbb{Z}^2]$ summand.  Using this, we show that W. Thurston's tripus is a homology handlebody of genus $2$ whose group is not very large.  It is known to be hyperbolic manifold manifold and hence is large by Agol's theorem. 

\newtheorem*{proptri}{Proposition~\ref{prop:tri}}
\begin{proptri}
Let $T$ be W. Thurston's Tripus manifold.  Then $\pi_1(T)$ is a weakly parafree group of rank $2$ that is large but not very large. 
\end{proptri}

It is unknown if there is a homology handlebody of genus $g\geq 3$ whose group is not very large. All of the examples that the authors have considered with $g\geq 3$ have been shown to be very large.  

\begin{question} If $Y$ is a homology handlebody of genus $g$ with $g\geq 3$, is $\pi_1(Y)$ very large?  More generally, if $G$ is a finitely presented group with $H_1(G)\cong \mathbb{Z}^g$ for $g\geq 3$ and $H_2(G)=0$, is $G$ very large? 
\end{question}

Often a homology handlebody group is very large.  In this case, we ask what is the maximal rank free group that arises as the quotient of the group.  
The \textbf{cut number of $M$}, $
c\left( M\right) $, is defined to be the maximal number of components of a
compact, orientable, $2$--sided surface $F$ properly embedded in $M$ such that $M\smallsetminus F$ is
connected.  Hence, for any $n\leq c\left( M\right) $, we can construct a
map $f : M\rightarrow \bigvee_{i=1}^{n}S^{1}$ such that the induced map on $
\pi _{1}$ is surjective.  That is, there exists a surjective map $f_{\ast
}: \pi _{1}\left( M\right) \twoheadrightarrow F\left( n\right) $, where $
F\left( n\right) $ is the free group of rank $n$.
Conversely, if we have any epimorphism $\phi : \pi _{1}\left( M\right)
\twoheadrightarrow F\left( n\right) $, then we can find a map $
f : M\rightarrow \bigvee_{i=1}^{n}S^{1}$ such that $f_{\ast }=\phi $. After
making the $f$ transverse to a non-wedge point $x_{i}$ on each $S^{1}$, $
F_i= f^{-1}\left( x_i\right) $ will give $n$ disjoint surfaces $F=\cup F_{i}$ with $
M\smallsetminus F$ connected. Hence one has the following elementary group-theoretic
characterization of $c\left( M\right)$, as in the closed case. 

\newtheorem*{jacothm}{Theorem (\cite[Theorem 2.1]{Ja72})}
\begin{jacothm}
The cut number $c\left( M\right) $ of $M$ is the maximal integer $n$ such that there is a surjective homomorphism $\phi : \pi _{1}\left( M\right) \twoheadrightarrow F\left(
n\right) $ onto the free group of rank $n$.
\end{jacothm}

The maximal rank of any free quotient of group is called its \textbf{co-rank}.   Hence the co-rank of $\pi_1(M)$ is the same as the cut number of $M$.
This is also referred to as the \textbf{inner rank} of $\pi_1(M)$ \cite{Lyndon} as well as the \textbf{non-commutative first Betti number} of $M$ \cite{AL86}. 
For each $g$, we construct an example of $3$-dimensional handlebody of genus $g$ whose groups maps onto a free group of rank $r(g)$ (roughly half of $g$) but no larger.  

\newtheorem*{mainthm}{Theorem~\ref{thm:bound}}
\begin{mainthm}
For each $g \geq 1$, there is a compact, connected, orientable $3$-dimensional homology handlebody $Y_g$ of genus $g$ with $c(Y_g) = r(g)$ where
$$
r(g) = \begin{cases}
\frac{g}{2} & \text{if $g$ is even}\\
\frac{g+1}{2} & \text{if $g$ is odd}.
\end{cases}
$$
\end{mainthm}

We expect that this may be optimal.  

\begin{question} If $Y$ is a compact, connected, orientable $3$-dimensional homology handlebody of genus $g$, is $c(Y)\geq r(g)$?
\end{question}

\subsection*{Acknowledgements} The first author would like to thank Michael Freedman for asking her about the cut number of homology handlebodies.  We would like to thank John Hempel, Neil Fullarton, and Alan Reid for helpful conversations.


\section{Main Theorem}\label{sec:bound}

The proof of Theorem~\ref{thm:bound} is almost immediate once we construct examples of homology handlebodies $Y_2$ and $Y_3$ of genus $2$ and $3$ repectively with $c(Y_2)=1$ and $c(Y_3)=2$.  We let $Y_2$ will be Thurston's tripus manifold $T$; we prove $\pi_1(T)$ is not very large in Subsection~\ref{sec:tri}.  We let $Y_3 = Y$, a particular genus-$3$ string link complement; we construct $Y$ and prove that $c(Y)=2$ in Subsection~\ref{sec:g3}.

\begin{theorem}
For each $g \geq 1$, there is a compact, connected, orientable $3$-dimensional homology handlebody $Y_g$ of genus $g$ with $c(Y_g) = r(g)$ where
$$
r(g) = \begin{cases}
\frac{g}{2} & \text{if $g$ is even}\\
\frac{g+1}{2} & \text{if $g$ is odd}.
\end{cases}
$$
\label{thm:bound}
\end{theorem}
\begin{proof}

For each integer $g \geq 1$ we shall construct a homology handlebody $Y_g$ of genus $g$ with $c(Y_g) = r(g)$.  Because $r(1) = 1$, we may choose $Y_1$ to be a handlebody.  Let $Y_2 = T$ and $Y_3 = Y$, where $T$ is Thurston's Tripus manifold (see Figure~\ref{fig:tripus}) and $Y$ is the string link complement from Figure~\ref{fig:g3}.  Then by Propositions \ref{prop:tri} and \ref{prop:g3}, $c(Y_2) = 1 = r(2)$ and $c(Y_3) = 2 = r(3)$.  If $g \geq 4$, we define $Y_g$ inductively as the boundary connected sum $Y_g = Y_{g-2} \natural Y_2$.  By Theorem 3.2 of \cite{Ja72}, the cut number is additive under boundary connected sum and so $c(Y_g) = c(Y_{g-2}) + 1$.  Since it is also the case that $r(g) = r(g-2) + 1$, it follows by induction that $c(Y_g) = r(g)$ for all $g \geq 0$.
\end{proof}

\subsection{Preliminary notions and lemmas}
Before we prove that $c(Y_2)=1$ and $c(Y_3)=2$, we shall determine some obstructions to a group having a non-abelian free quotient.  To do this, we look at the first homology (and relative first homology) of free abelian covering spaces of the manifolds.  We note that the co-rank is known to be algorithmically computable \cite{Ma82, Ra95} for finitely presented groups;  however, the algorithm seems difficult to use in practice.   

Let $X$ be a path connected topological space and $\phi:\pi_1(X) \rightarrow \Gamma$ be a surjective group homomorphism.  We define $X_\phi \xrightarrow{\pi} X$ to be the covering space of $X$ corresponding to $\phi$.  Recall that the group of deck translations of $X_\phi \rightarrow X$ is identified with $\Gamma$ making $H_1(X_\phi)$ into a left $\mathbb{Z}\Gamma$-module.  Moreover, for any group $G$ and surjective homomorphism $\phi:G \rightarrow \Gamma$, $\frac{Ker(\phi)}{[Ker(\phi),Ker(\phi)]}$ has the structure of a left $\mathbb{Z}\Gamma$-module where the module action is given as follows.  Let $\gamma \in \Gamma$ and $g\in Ker(\phi)$.  Then $\gamma \left[g\right] := \left[hgh^{-1}\right]$ for any $h\in G$ such that $\phi(h)=\gamma$ (here $\left[g\right]$ denotes the equivalence class of $g$).  In the case that $G=\pi_1(X)$, we have that $H_1(X_\phi)$ is isomorphic to $\frac{Ker(\phi)}{[Ker(\phi),Ker(\phi)]}$ as left $\mathbb{Z}\Gamma$-modules.

\begin{lemma}\label{lem:summand} Let $G$ be a group and $\Gamma = \mathbb{Z}^m$ with $m \in \{1,2\}$.  If $G$ is very large, then there exists a surjective homomorphism $\phi: G \rightarrow \Gamma$ such that 
$$ \frac{Ker(\phi)}{[Ker(\phi),Ker(\phi)]}  \cong  \mathbb{Z}\Gamma \oplus A$$ for some left $\mathbb{Z}\Gamma$-module $A$. In particular, if $M$ is a compact, connected, orientable $3$-manifold with $\pi_1(M)$ very large, then there is a surjective homomorphism $\phi: \pi_1(M) \rightarrow \Gamma$ such that $H_1(M_\phi)$ has a $\mathbb{Z}\Gamma$ summand. 
\end{lemma}

\begin{proof}Since $G$ is very large, there is a surjective homomorphism $\psi^\prime: G \twoheadrightarrow F$ where $F$ is the free group of rank $2$. Let $\psi:F \twoheadrightarrow \Gamma$ be a surjective homomorphism and $\phi = \psi \circ \psi^\prime$.  Then $\psi$ induces a surjective $\mathbb{Z}\Gamma$-module homomorphism $$\psi: \frac{Ker(\phi)}{[Ker(\phi),Ker(\phi)]} \twoheadrightarrow \frac{Ker(\psi)}{[Ker(\psi),Ker(\psi)]}.$$  Since $\frac{Ker(\psi)}{[Ker(\psi),Ker(\psi)]}   \cong  \mathbb{Z}\Gamma$ is a free $\mathbb{Z}\Gamma$ module, the result follows. 
\end{proof}

We note that in the previous proof that if $m\geq 3$ then $\frac{Ker(\psi)}{[Ker(\psi),Ker(\psi)]}$ is no longer a free module; in fact, it is not even projective.  Thus, to generalize the previous lemma, we work with relative homology. Given $p \in M$, let $p_{\phi}$ denote the preimage $\pi^{-1}(p) \subset M_{\phi}$.  The relative homology group $H_1(M_{\phi},p_{\phi})$ has the structure of a left $\mathbb{Z}\Gamma$-module as before, coming from group of deck translations acting on the pair $(M_{\phi},p_{\phi})$.  We have the following generalization of Lemma \ref{lem:summand}, which no longer restricts us to using free abelian covering spaces.
\begin{lemma}\label{lem:relative}
Let $M$ be a compact, connected, orientable $3$-manifold, let $p \in M$, and let $\Gamma$ be a quotient of $F(n)$, the non-abelian free group of rank $n$ for some $n \leq c(M)$.  Then there is a surjective homomorphism $\phi: \pi_1(M) \rightarrow \Gamma$ such that
	\[H_1(M_{\phi},p_\phi) \cong \mathbb{Z}\Gamma^n \oplus A \] 
	for some left $\mathbb{Z}\Gamma$-module $A$.
\end{lemma}

\begin{proof}Since $c(M) \geq n$, there is an epimorphism $\psi^\prime:\pi_1(M,p) \twoheadrightarrow F(n)$, the free group of rank $n$.  Letting $W$ denote the $n$-fold wedge of circles and $w\in W$ the wedge point, there exists a map $f:M \rightarrow W$ such that $f(p) = w$ and such that the induced map on $\pi_1$ is $\psi^\prime$.  Let 
$\psi: \pi_1(W,w) \rightarrow \Gamma$ be a surjective map and $\phi = \psi \circ \psi^\prime$.    The map $f$ can be lifted to a map $g:M_{\phi} \rightarrow W_\psi$ sending $p_\phi$ to $w_\psi$.

We examine the following commutative diagram of modules and module homomorphisms over $\mathbb{Z}\Gamma$, where the rows are portions of the long exact sequences of the pairs and all of the vertical maps are the module homomorphisms on homology induced by $g$.

$$
\begin{tikzcd}
H_1(M_{\phi}) \arrow[d, "g_1"] \arrow[r] & H_1(M_{\phi},p_{\phi})  \arrow[d,"g_2"] \arrow[r] & H_0(p_{\phi})  \arrow[d,"g_3"] \arrow[r] & H_0(M_{\phi}) \arrow[d,"g_4"]\\
H_1(W_{\psi}) \arrow[r] & H_1(W_{\psi},w_\psi) \arrow[r] & H_0(w_\psi) \arrow[r] & H_0(W_\psi)
\end{tikzcd}
$$

Note first that $g_4$ is an isomorphism and $g_3$ is an epimorphism.  Also, since $\psi^\prime$ is an epimorphism it follows that $g_1$ is as well.  Together these facts imply that $g_2$ is an epimorphism as well.  Since $H_1(W_\psi,w_\psi)$ is a free module of rank $n$, the result follows. 

\end{proof}

The following easy lemma is most likely well-known but we include it for completeness.  This will be used to prove that the modules arising in the proofs Propositions \ref{prop:tri} and \ref{prop:g3} are torsion free. 

\begin{lemma}\label{lem:gcd}Suppose $R$ is a unique factorization domain and $M$ is a left module over $R$ with a presentation of the form 
	$$\left<\alpha_1, \ldots, \alpha_n \left|\right. p_1 \alpha_1 + \ldots p_n \alpha_n\right>$$ where $p_k\in R$ for each $k$.  Then:
	\begin{enumerate}[(i)]
	\item If $p_i$ and $p_j$ are relatively prime for some $i \neq j$, then $M$ is torsion-free.
	\item If there is a common factor different than $1$ which is mutually shared by all of $p_1, \ldots, p_n$, then $M$ is not torsion-free.
	\end{enumerate}
	In particular, for $n = 2$, $M$ is torsion free if and only if $p_1$ and $p_2$ are relatively prime.
\end{lemma}
	
\begin{proof}
		Statement $(ii)$ is obvious.  We proceed to prove statement $(i)$.  Suppose $M$ has a torsion element, $f_1 \alpha_1 + \ldots + f_n\alpha_n \neq 0$ with $f_k \in R$ for each $k$.  Then there exists $r,s \in R$ (with $s$ non-zero and not a unit) such that $s(f_1 \alpha_1 + \ldots + f_n\alpha_n) = r( p_1 \alpha_1 + \ldots p_n \alpha_n)$ in the free $R$ module generated by $\alpha$ and $\beta$. Hence $sf_k = rp_k$ for each $k$.  We can assume that $s$ and $r$ have no common factors, otherwise, we could reduce them.  So $s$ much divide $p_k$ for each $k$.  Hence no pair $p_i$, $p_j$ can be relatively prime.
\end{proof}


\subsection{Thurston's tripus manifold has cut number equal to 1}\label{sec:tri}

Let $T$ be W. Thurston's tripus manifold.   That is, $T$ is the complement of the three arcs in $S^2 \times I$ as shown in Figure~\ref{fig:tripus}.  Smoothly embed $S^2 \times I$ in $S^3$.  Then we see that $T$ is the complement of an  embedding of a $\theta$ graph in $S^3$ (where the exterior of $S^2 \times I$ can be identified with with the two vertices of the spatial graph). Note that the tripus is a hyperbolic manifold with totally geodesic boundary and is also a genus $2$ homology handlebody.

	\begin{figure}
	\includegraphics{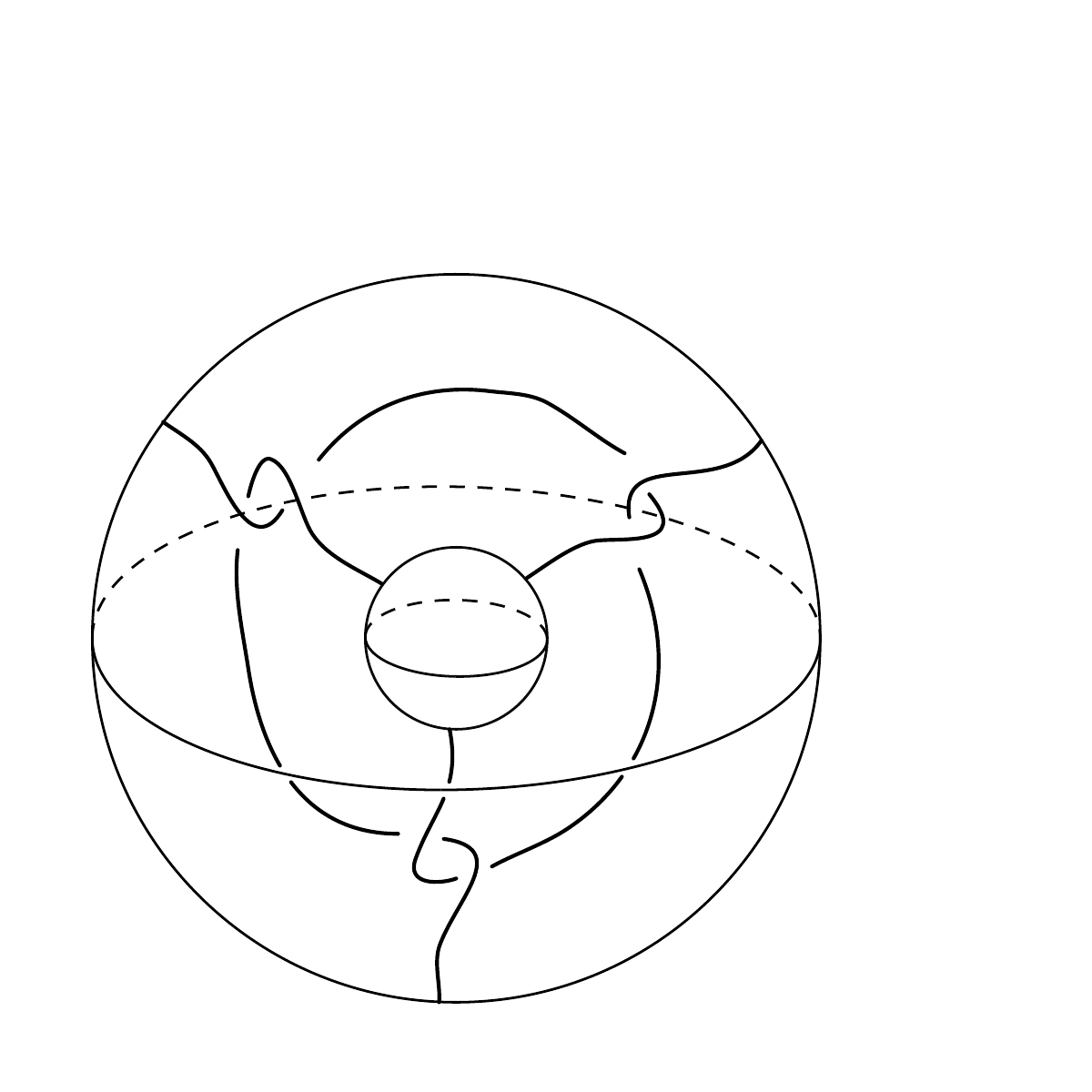}
	\caption{Thurston's Tripus, a hyperbolic manifold with totally geodesic boundary}
	\label{fig:tripus}
	\end{figure}

\begin{proposition}
Let $T$ be W. Thurston's Tripus manifold.  Then $\pi_1(T)$ is a weakly parafree group that is large but not very large. 
\label{prop:tri}
\end{proposition}
	\begin{proof} Since $T$ is a hyperbolic manifold with totally geodesic boundary \cite[Chapter 3]{WTbook}, it is irreducible and has incompressible boundary.  Thus by \cite{CLR}, $\pi_1(T)$ has a finite index subgroup with a non-abelian free quotient.  That is, $\pi_1(T)$ is large.  
		
		To show that $\pi_1(T)$ is not very large, we use Lemma~\ref{lem:summand}.  We begin by computing the first homology of the universal abelian cover of $T$.
		We can use the Wirtinger presentation to write down a presentation for $\pi_1(T)$:

	$$\pi_1(T) \cong \left<a,b,c,d,e,f \left|\right. ace, bdf, (fc)b(fc)^{-1}a^{-1}, (be)d(be)^{-1}c^{-1}, (da)f(da)^{-1}e^{-1} \right>.$$ 
	Using the relation that $e=(ac)^{-1}$ and $f=(bd)^{-1}$, we get a simplified presentation:
	 $$\pi_1(T) \cong \left<a,b,c,d \left|\right. (d^{-1}b^{-1}c)b(d^{-1}b^{-1}c)^{-1}a^{-1}, (bc^{-1}a^{-1})d(bc^{-1}a^{-1})^{-1}c^{-1}, (da)d^{-1}b^{-1}(da)^{-1}ac \right>.$$ 
	 In fact, the third relation above can be derived from the first two, and thus we can use the presentation
	 $$\pi_1(T) \cong \left<a,b,c,d \left|\right. (d^{-1}b^{-1}c)b(d^{-1}b^{-1}c)^{-1}a^{-1}, (bc^{-1}a^{-1})d(bc^{-1}a^{-1})^{-1}c^{-1} \right>.$$ 
	 
	  We now change the generating set to $\left\{a, B, c, D\right\}$ where $B=ba^{-1}$ and $D=dc^{-1}$.  Let $\phi: \pi_1(T) \rightarrow H_1(T)$ be the abelianization map.  We note that the image of two of the generators, $a$ and $c$,  under $\phi$ give a basis for $H_1(T)$.  To make the notation easier, we will call $\phi(a)$ (respectively $\phi(c)$), $a$ (respectively $c$).  Let $\Theta = aca^{-1}c^{-1}$.  Since $B$ and $D$ are trivial under $\phi$, they lift to $T_\phi$ and the set  $\{ B, D,  \Theta \}$ generates $H_1(T_\phi)$ as a $\mathbb{Z}[H_1(T)]$-module.  It is straightforward to write down a presentation matrix of $H_1(T_\phi)$ as a $\mathbb{Z}[a^{\pm 1},c^{\pm 1}]$-module (here the rows are the relations and the columns correspond to $B$, $D$, and $\Theta$):
	$$ \begin{bmatrix}
	a + c - 1 & a(a-1) & a- 1\\
	c(1-c) & 1 & c- 1
	\end{bmatrix}.
	$$
	Using row and column operations, we can find a simpler presentation:
	\begin{align*}
	\begin{bmatrix}
	a +c - 1 & a(a-1) & a- 1\\
	c(1-c) & 1 & c- 1
	\end{bmatrix} &\sim
	\begin{bmatrix}
	ac+a-1 & a^2-a & a-1\\
	0 & 1 & c-1
	\end{bmatrix}\\
	&\sim\begin{bmatrix}
	ac+a-1 & a^2-a & -a^2c-a^2-ac+a\\
	0 & 1 & c-1
	\end{bmatrix}\\
	&\sim \begin{bmatrix}
	ac+a-1 & 0 & -2a^2c \\
	0 & 1 & c-1
	\end{bmatrix}\\
	&\sim \begin{bmatrix}
	ac+a-1 & 2a^2c
	\end{bmatrix}
	\end{align*}
Since $gcd(ac+a-1,2a^2c) = 1$, Lemma \ref{lem:gcd} implies that $H_1(T_\phi)$ is torsion-free as a $H_1(T)$-module.  We claim that $H_1(T_\phi)$ is not a free module.  To see this, consider the ring homomorphism 
$\xi:  \mathbb{Z}[H_1(T)] \rightarrow \mathbb{Z}_2[a^{\pm 1},c^{\pm 1}]$ where we reduce the coefficient mod $2$.  This map endows $ \mathbb{Z}_2[a^{\pm 1},c^{\pm 1}]$ with the structure of a right $\mathbb{Z}[H_1(T)]$-module.
If $H_1(T_\phi)$ were free as $\mathbb{Z}[H_1(T)]$-module then the left $\mathbb{Z}_2[a^{\pm 1},c^{\pm 1}]$-module
$$\mathbb{Z}_2[a^{\pm 1},c^{\pm 1}] \otimes_{\mathbb{Z}[H_1(T)]} H_1(T_\phi)$$
 would be free as well.  However, this tensor product module is presented by the matrix $[ac+a+1,0]$, and is thus not free.

Suppose that $\pi_1(T)$ is very large.  We note that there is a unique covering space of $T$ with deck transformation group isomorphic to $\mathbb{Z}^2$, up to covering space isomorphism.  Hence by Lemma~\ref{lem:summand}, $H_1(T_\phi)\cong \mathbb{Z}[H_1(T)] \oplus A$.  Since the rank of $H_1(T_\phi)$ (as a $\mathbb{Z}[H_1(T)]$-module) in our case is $1$, it follows that $A$ must have rank $0$ and hence is a torsion module.  However, we know that $H_1(T_\phi)$ is torsion-free so $A=0$.  This contradicts the fact that $H_1(T_\phi)$ is not a free module and the result follows.

	\end{proof}
%
%


\subsection{A genus-3 homology handlebody with cut number equal to $2$}\label{sec:g3}

\begin{proposition}
There is a genus $3$ homology handlebody $Y$ with $c(Y) = 2$.
\label{prop:g3}
\end{proposition}

\begin{proof}
\begin{figure}[h]
   \begingroup%
     \makeatletter%
     \providecommand\color[2][]{%
       \errmessage{(Inkscape) Color is used for the text in Inkscape, but the package 'color.sty' is not loaded}%
       \renewcommand\color[2][]{}%
     }%
     \providecommand\transparent[1]{%
       \errmessage{(Inkscape) Transparency is used (non-zero) for the text in Inkscape, but the package 'transparent.sty' is not loaded}%
       \renewcommand\transparent[1]{}%
     }%
     \providecommand\rotatebox[2]{#2}%
     \ifx\svgwidth\undefined%
       \setlength{\unitlength}{113.98312855bp}%
       \ifx\svgscale\undefined%
         \relax%
       \else%
         \setlength{\unitlength}{\unitlength * \real{\svgscale}}%
       \fi%
     \else%
       \setlength{\unitlength}{\svgwidth}%
     \fi%
     \global\let\svgwidth\undefined%
     \global\let\svgscale\undefined%
     \makeatother%
     \begin{picture}(1,1.85116849)%
       \put(0,0){\includegraphics[width=\unitlength,page=1]{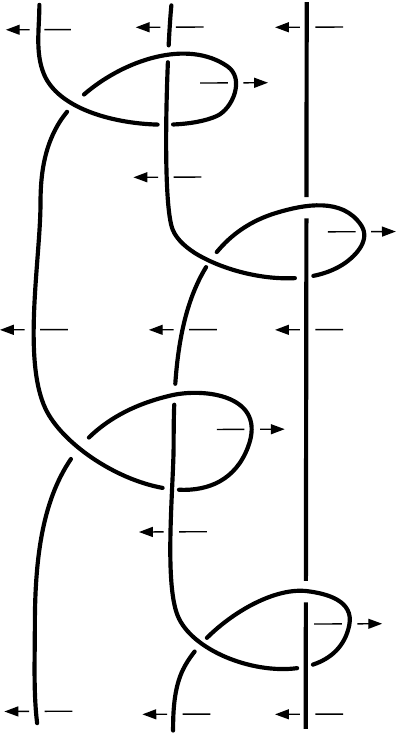}}%
       \put(0.13736498,1.80426085){\color[rgb]{0,0,0}\makebox(0,0)[lb]{$x$}}%
       \put(0.47393646,1.80211707){\color[rgb]{0,0,0}\makebox(0,0)[lb]{$a$}}%
       \put(0.83408931,1.80426085){\color[rgb]{0,0,0}\makebox(0,0)[lb]{$r$}}%
       \put(0.46750513,1.42052649){\color[rgb]{0,0,0}\makebox(0,0)[lb]{$b$}}%
       \put(0.52753063,1.04322349){\color[rgb]{0,0,0}\makebox(0,0)[lb]{$d$}}%
       \put(0.94556532,1.30261926){\color[rgb]{0,0,0}\makebox(0,0)[lb]{$c$}}%
       \put(0.84052069,1.03893583){\color[rgb]{0,0,0}\makebox(0,0)[lb]{$s$}}%
       \put(0.66258801,0.79883382){\color[rgb]{0,0,0}\makebox(0,0)[lb]{$w$}}%
       \put(0.46964896,0.44725602){\color[rgb]{0,0,0}\makebox(0,0)[lb]{$e$}}%
       \put(0.92412767,0.2993362){\color[rgb]{0,0,0}\makebox(0,0)[lb]{$f$}}%
       \put(0.83194559,0.05923408){\color[rgb]{0,0,0}\makebox(0,0)[lb]{$t$}}%
       \put(0.48679906,0.07209684){\color[rgb]{0,0,0}\makebox(0,0)[lb]{$g$}}%
       \put(0.13950875,0.07852812){\color[rgb]{0,0,0}\makebox(0,0)[lb]{$u$}}%
       \put(0.12878993,1.03893583){\color[rgb]{0,0,0}\makebox(0,0)[lb]{$z$}}%
       \put(0.62185633,1.564){\color[rgb]{0,0,0}\makebox(0,0)[lb]{$y$}}%
     \end{picture}%
   \endgroup%
\caption{The homology handlebody $Y$ is the complement of this pure 3-string link.  Generators of $\pi_1(Y)$ are labelled.}
\label{fig:g3}
\end{figure}

Let $Y$ denote the homology handlebody that is the complement in $D^2 \times I$ of the three-component pure string link show in Figure \ref{fig:g3}.  Removing the middle string leaves of with the trivial string link whose complement is a handlebody.  Hence there is a surjective map $\pi_1(Y) \rightarrow F(2)$, hence $c(Y)\geq 2$.

A presentation of $\pi_1(Y)$ can be computed via the Wirtinger presentation:

$$\pi_1(Y) \cong \left<\begin{array}{c|c}
a, b, c, d, e, f, g, &
zuz^{-1} w^{-1}, eze^{-1}w^{-1}, xzx^{-1}y^{-1}, bxb^{-1}y^{-1}, fsf^{-1}t^{-1}, crc^{-1}s^{-1}, \\
r, s, t, x, y, z, u, w &
ege^{-1}f^{-1}, tet^{-1}f^{-1}, wdw^{-1}e^{-1}, bdb^{-1}c^{-1}, sbs^{-1}c^{-1}, yay^{-1}b^{-1}\end{array}\right>$$

First notice that the generators $a,g,r,u$ each appear in only one relation, and can thus be eliminated along with those relations to produce the following smaller presentation:

$$\pi_1(Y) \cong \left<\begin{array}{c|c}
b, c, d, e, f, &
bxb^{-1}y^{-1}, xzx^{-1}y^{-1}, eze^{-1}w^{-1}, tet^{-1}f^{-1},\\
s, t, x, y, z, w &
sbs^{-1}c^{-1},  fsf^{-1}t^{-1},bdb^{-1}c^{-1},wdw^{-1}e^{-1} \end{array} \right>$$

We shall obtain an even simpler presentation by combining relations.  The first four relations imply that
$$ x = b^{-1}yb = b^{-1}xzx^{-1}b = b^{-1}xe^{-1}wex^{-1}b = b^{-1}xt^{-1}f^{-1}twt^{-1}ftx^{-1}b$$
Similarly, the last five relations imply that $b = f^{-1}t^{-1}fbw^{-1}t^{-1}ftwb^{-1}f^{-1}tf$.  These observations lead us to a presentation with seven generators and only two relations:
$$ \pi_1(Y) \cong \left< b,f,t,x,w \ | \ b^{-1}xt^{-1}f^{-1}twt^{-1}ftx^{-1}bx^{-1}, f^{-1}t^{-1}fbw^{-1}t^{-1}ftwb^{-1}f^{-1}tfb^{-1} \right>$$

Consider the abelianization map $\phi:\pi_1(Y) \rightarrow H_1(Y)$.  Note that $\phi(x)=\phi(w)$ and $\phi(b)=\phi(f)$, so that the images of $b$, $t$, and $x$ give a basis for $H_1(Y)$.  Considering the universal abelian cover $Y_{\phi}$, we choose a basepoint $p \in Y$ and let $p_\phi = \phi^{-1}(p) \in Y_{\phi}$.  Using the Fox differential calculus, we can obtain a presentation for the structure of $H_1(Y_{\phi},p_{\phi})$ as a left module over $\mathbb{Z}[H_1(Y)] \cong \mathbb{Z}[b^{\pm 1}, t^{\pm 1}, x^{\pm 1}]$.  With respect to the generators $b,f,t,x,w$, a presentation matrix is

\begin{equation}
\begin{bmatrix}
btx-bt & x^2-x &  bx^2 - bx -x^2 + x & bt - b^2t - btx & tx \\
btx - b^2tx - bt^2x & bt^2x - t^2x - btx + tx + b^2 &  btx - tx + b^3 - b^2 & 0& b^3t - b^2t
\end{bmatrix}
\label{g3-1}
\end{equation}

We first claim that this module is not free.  We use $\xi: H_1(Y) \rightarrow  \mathbb{Z}_3[b^{\pm 1}, t^{\pm 1}, x^{\pm 1}]$, given by reducing coefficients modulo $3$, to endow $\mathbb{Z}_3[b^{\pm 1}, t^{\pm 1}, x^{\pm 1}]$ with the structure of a right $\mathbb{Z}[H_1(Y)]$-module.  If $H_1(Y_{\phi},p_{\phi})$ were a free $\mathbb{Z}[H_1(Y)]$-module, then the left $\mathbb{Z}_3[b^{\pm 1}, t^{\pm 1}, x^{\pm 1}]$-module
$$\mathbb{Z}_3[b^{\pm 1}, t^{\pm 1}, x^{\pm 1}] \otimes_{\mathbb{Z}[H_1(Y)]}  H_1(Y_{\phi},p_{\phi}) $$
would also be free.

Recall that given a rank $r$ module, $N$, over a multivariable Laurent polynomial ring $R$, is projective if and only if the $r^{th}$ elementary ideal $E_r(N)$ of $N$ is equal to all of $R$.  We used the Magma Computational Algebra System \cite{Magma} to verify that the ideal generated by the $2\times 2$ minors of the above matrix, viewed modulo $3$, is indeed proper in $\mathbb{Z}_3[b^{\pm 1}, t^{\pm 1}, x^{\pm 1}]$.  Therefore the $\mathbb{Z}_3[b^{\pm 1}, t^{\pm 1}, x^{\pm 1}]$-module
$$\mathbb{Z}_3[b^{\pm 1}, t^{\pm 1}, x^{\pm 1}] \otimes_{\mathbb{Z}[H_1(Y)]}  H_1(Y_{\phi},p_{\phi}) $$
is not projective and hence not free.  It follows that $H_1(Y_{\phi},p_{\phi})$ is not a free $\mathbb{Z}[H_1(Y)]$-module.

Our Magma source code can be found below.  Note that here $B = b^{-1}$, $T = t^{-1}$, and $X = x^{-1}$.

\vspace{1pc}
\hrule
\vspace{1pc}
\begin{lstlisting}
P<x,X,b,B,t,T> := PolynomialRing(GF(3),6);
Q<x,X,b,B,t,T> := quo<P | x*X-1, t*T-1, b*B-1>;
M := Matrix(Q, 2, 5, [b*t*(x-1), x*(x-1), x*(b-1)*(x-1),
		  b*t*(1-x-b), t*x,b*t*x*(1-b-t),
		  -t^2*x+t*x+b^2-b*t*x+b*t^2*x,
		  -t*x-b^2+b^3+b*t*x,0, b^2*t*(b-1)]);
Mi := Minors(M, 2);
I := ideal<Q | Mi>;
IsProper(I)
\end{lstlisting}
\vspace{1pc}
\hrule
\vspace{1pc}

We now claim that the module $H_1(Y_{\phi},p_{\phi})$ is torsion-free.  After multiplying the second row of the matrix (\ref{g3-1}) by $x$, adding $b^2-b^3$ times the first row to the second row, eliminating a generator and relation, and rescaling several entries by units we obtain a smaller presentation matrix:

\begin{equation}
\begin{bmatrix}
-b^3x + b^3 + b^2x -bx^2 - tx^2 - b^2 + x^2\\
-b^3x + bt^2x + b^3 + b^2x - btx - t^2x + tx\\
-b^4x + b^4 + 2b^3x-b^3-b^2x + btx-tx\\
b^2 + bx - 2b - x + 1
\label{g3-2}
\end{bmatrix}^T
\end{equation}

Now consider the ring homomorphism $\eta:\mathbb{Z}[H_1(Y)] \rightarrow \mathbb{Z}[ t^{\pm 1}, x^{\pm 1}]$ generated by the homomorphism $H_1(Y) \rightarrow <t,x>$ sending $b$ and $t$ to $t$ and $x$ to $x$.  As before $\eta$ induces a left $\mathbb{Z}[ t^{\pm 1}, x^{\pm 1}]$-module structure on
$$ \mathbb{Z}[ t^{\pm 1}, x^{\pm 1}] \otimes_{\mathbb{Z}[H_1(Y)]} H_1(Y_{\phi},p_{\phi})$$
which has torsion if $H_1(Y_{\phi},p_{\phi})$ has torsion.  We obtain a presentation matrix for the above tensor module by setting $b = t$ in (\ref{g3-2}) and rescaling several entries by units:
$$
\begin{bmatrix}
-t^3x + t^3 + t^2x -2tx^2 - t^2 + x^2\\
t^2 - tx + x\\
-t^3x + t^3 + 2t^2x - t^2 - x\\
t^2 + tx - 2t - x + 1
\end{bmatrix}^T
$$
We verified in Magma that the first two entries of this matrix (in fact, any two entries) are relatively prime.  By Lemma \ref{lem:gcd}, the module is torsion-free.

Thus,  $H_1(Y_{\phi},p_{\phi})$ is a torsion-free module of rank $3$.  Suppose that $\pi_1(Y)$ maps onto a free group of rank $3$.  Since $\beta_1(Y)=3$, there is a unique $\mathbb{Z}^3$ covering space (up to covering space isomorphism) so by Lemma~\ref{lem:relative}, $H_1(Y_{\phi},p_{\phi}) \cong \mathbb{Z}[H_1(Y)]^3 \oplus A$ for some $A$.  Since $H_1(Y_{\phi},p_{\phi})$ is rank $3$, it follows that $A$ must be a torsion module.  However $H_1(Y_{\phi},p_{\phi})$ is torsion free so $A=0$. This contradicts the fact that $H_1(Y_{\phi},p_{\phi})$ is not free and it follows that $c(Y) = 2$.  
\end{proof}

\bibliographystyle{alpha}
\def\MR#1{}
\bibliography{hhbib}

\end{document}